\documentclass[11pt]{amsart} 
\usepackage{amsfonts} 
\usepackage{graphicx}
\usepackage{color}
\newtheorem{theorem}{Theorem}[section]

\newtheorem{lemma}[theorem]{Lemma}
\newtheorem{corollary}[theorem]{Corollary}
\newtheorem{definition}[theorem]{Definition}

\renewenvironment{proof}{{\noindent\bf Proof.}}{\hfill $\Box$\par\vskip3mm}

\begin{document}

\def\reals{{\mathbb R}}
 \def\ch{{\mathcal H}}
 \def\cA{{\mathcal A}}
 \def\cD{{\mathcal D}}
 \def\cK{{\mathcal K}}
 \def\cC{{\mathcal C}}
 \def\cN{{\mathcal N}}
 \def\cR{{\mathcal R}}
 \def\cS{{\mathcal S}}
 \def\cT{{\mathcal T}}
 \def\cV{{\mathcal V}}
 %notation for cabling torus
 \def\ta{{\mathcal T}_{\subset}}
 \def\cI{{\mathcal I}}
 %notation for cabling operation
 \def\bC{{\bf C}}
 %notation for braid axis
 \def\axis{{\bf A}}
 %notation for braid fibration
 \def\fibr{{\bf H}}
 %a-arcs
 \def\ba{{\bf a}}
 %b-arcs
 \def\bb{{\bf b}}
 %c-arcs
 \def\bc{{\bf c}}
 %e-arcs
 \def\be{{\bf e}}
 \def\d{{\delta}} 
 \def\ci{{\circ}} 
 \def\e{{\epsilon}} 
 \def\l{{\lambda}} 
 \def\L{{\Lambda}} 
 \def\m{{\mu}} 
 \def\n{{\nu}} 
 \def\o{{\omega}} 
 \def\s{{\sigma}} 
 \def\v{{\varphi}} 
 \def\a{{\alpha}} 
 \def\b{{\beta}} 
 \def\p{{\partial}} 
 \def\r{{\rho}} 
 \def\ra{{\rightarrow}} 
 \def\lra{{\longrightarrow}} 
 \def\g{{\gamma}} 
 \def\D{{\Delta}} 
 \def\La{{\Leftarrow}} 
 \def\Ra{{\Rightarrow}} 
 \def\x{{\xi}} 
 \def\c{{\mathbb C}} 
 \def\z{{\mathbb Z}} 
 \def\2{{\mathbb Z_2}} 
 \def\q{{\mathbb Q}} 
 \def\t{{\tau}} 
 \def\u{{\upsilon}} 
 \def\th{{\theta}} 
 \def\la{{\leftarrow}} 
 \def\lla{{\longleftarrow}} 
 \def\da{{\downarrow}} 
 \def\ua{{\uparrow}} 
 \def\nwa{{\nwtarrow}} 
 \def\swa{{\swarrow}} 
 \def\nea{{\netarrow}} 
 \def\sea{{\searrow}} 
 \def\hla{{\hookleftarrow}} 
 \def\hra{{\hookrightarrow}} 
 \def\sl{{SL(2,\mathbb C)}} 
 \def\ps{{PSL(2,\mathbb C)}} 
 \def\qed{{\hfill$\diamondsuit$}} 
 \def\pf{{\noindent{\bf Proof.\hspace{2mm}}}} 
 \def\ni{{\noindent}} 
 \def\sm{{{\mbox{\tiny M}}}} 
 \def\sc{{{\mbox{\tiny C}}}}

\title{Studying uniform thickness I:\\Legendrian simple iterated torus knots}

\author{Douglas J. LaFountain}
%\thanks{{\em Keywords:}  }
%\thanks{$^*$ This paper was partially supported by ....}
%\thanks{$^2$}
%\date{May 18, 2009}

\begin{abstract}
We prove that the class of topological knot types that are both Legendrian simple and satisfy the uniform thickness property (UTP) is closed under cabling.  An immediate application is that all iterated cabling knot types that begin with negative torus knots are Legendrian simple.  We also examine, for arbitrary numbers of iterations, iterated cablings that begin with positive torus knots, and establish the Legendrian simplicity of large classes of these knot types, many of which also satisfy the UTP.  In so doing we obtain new necessary conditions for both the failure of the UTP and Legendrian non-simplicity in the class of iterated torus knots, including specific conditions on knot types. 
\end{abstract}

\maketitle

\section{Introduction}

In this paper we begin a general study of the {\em uniform thickness property} (UTP) in the context of iterated torus knots that are embedded in $S^3$ with the standard tight contact structure.  Our goal in this study will be to determine the extent to which iterated torus knot types fail to satisfy the UTP, and the extent to which this failure leads to cablings that are Legendrian or transversally non-simple.  The specific goal of this note is to address both questions by establishing new necessary conditions for the failure of the UTP, as well as new necessary conditions for slopes of cablings that are Legendrian non-simple.  In the process we will show that, in some sense, most iterated torus knot types are Legendrian simple, and many satisfy the UTP, including many iterated cablings that begin with knots which fail the UTP.

Specifically, we will begin by showing that the class of knots that are both Legendrian simple and satisfy the UTP is closed under cabling, and hence all iterated cablings that begin with negative torus knots are Legendrian simple.  We will then study, for arbitrary numbers of iterations, iterated cablings that begin with positive torus knots, and demonstrate the Legendrian simplicity of many of these knot types, some of which also satisfy the UTP.  Our analysis will result in a precise class of iterated torus knot types that may fail the UTP, as well as the identification of many solid tori representatives that may fail to thicken.  We will also obtain a precise class of iterated torus knots that may be Legendrian non-simple.  A forthcoming note, {\em Studying uniform thickness II}, will then more directly address the related problems of determining whether these two classes indeed fail the UTP and are Legendrian non-simple.  

To bring the above goals into focus, we recall the definition of the {\em uniform thickness property} as given by Etnyre and Honda \cite{[EH1]}.  For a knot type $K$, define the {\em contact width} of $K$ to be

\begin{equation}
w(K)=\textrm{sup}\frac{1}{\textrm{slope}(\Gamma_{\partial N})}
\end{equation}

In this equation the $N$ are solid tori having representatives of $K$ as their cores, and $\textrm{slope}(\Gamma_{\partial N})$ refers to the slope of the {\em dividing curves} on the convex torus $\partial N$.  Slopes are measured using the preferred framing coming from a Seifert surface for $K$, and slopes are calculated so that the longitude has slope $\infty$; the supremum is taken over all solid tori $N$ representing $K$ where $\partial N$ is convex.  Any knot type $K$ satisfies the inequality $\overline{tb}(K) \leq w(K) \leq \overline{tb}(K) + 1$, where $\overline{tb}$ is the maximal Thurston-Bennequin number for $K$.

A knot type $K$ satisfies the UTP if the following hold:

\begin{itemize}
\item[1.] $\overline{tb}(K)=w(K)$.
\item[2.] Every solid torus $N$ representing $K$ can be thickened to a standard neighborhood of a maximal $tb$ Legendrian knot.
\end{itemize}

Using this definition, Etnyre and Honda identified necessary conditions for the existence of Legendrian non-simple iterated torus knot types \cite{[EH1]}.  Specifically, they showed that if all iterated torus knots were to satisfy the UTP, then they would all be Legendrian simple;  hence if some iterated torus knot fails to be Legendrian simple, then there must exist an iterated torus knot which fails the UTP.  They subsequently established that the $(2,3)$ torus knot fails the UTP and indeed has a cabling which is Legendrian non-simple, namely the $((2,3),(2,3))$ iterated torus knot.  They also established, for arbitrary numbers of iterations, iterated torus knots that are Legendrian simple, where at each iteration the knot type satisfies the UTP, and cabling fractions $P/q$ are less than the contact width.

In this note, we extend Etnyre and Honda's work; we begin by proving the following theorem:

\begin{theorem}
\label{UTP theorem}
Let $K$ be a topological knot type.  If $K$ is Legendrian simple and satisfies the UTP, then all of its cablings are Legendrian simple and satisfy the UTP.
\end{theorem}

Most of the content of this theorem was proved by Etnyre and Honda in Theorems 1.1 and 1.3 in \cite{[EH1]}; we prove the satisfaction of the UTP for cabling fractions $P/q$ that are greater than the contact width.  As an immediate consequence, using the fact that negative torus knots are Legendrian simple and satisfy the UTP \cite{[EH1],[EH2]}, we have the following result:

\begin{corollary}
\label{Cablings of negative torus knots}
All iterated cabling knot types that begin with negative torus knots are Legendrian simple; that is, if $K_r = ((P_1,q_1),...,(P_r,q_r))$ is an iterated torus knot type where $(P_1,q_1)$ is a negative torus knot, then $K_r$ is Legendrian simple.
\end{corollary}

We then undertake an analysis of iterated cablings that begin with positive torus knots, and identify, for arbitrary numbers of iterations, Legendrian simple classes of such iterated torus knots.  In order to obtain a precise statement of these other results, we will first need to recall and introduce some terminology.  However, at this point the reader may wish to look ahead to Figure \ref{fig:SchematicNewestB2}, where in graphical form we combine Etnyre and Honda's results with ours to provide a summary of what is known concerning the uniform thickness and Legendrian classification of iterated torus knots.

Recall that for Legendrian knots embedded in $S^3$ endowed with the standard tight contact structure, there are two classical invariants of Legendrian isotopy classes, namely the Thurston-Bennequin number, $tb$, and the rotation number, $r$.  For a given topological knot type, we can represent Legendrian isotopy classes by points on a grid whose horizontal axis plots values of $r$ and whose vertical axis plots values of $tb$.  This plot takes the visual form of a {\em Legendrian mountain range}.  For a given topological knot type, if the ordered pair $(r,tb)$ completely determines the Legendrian isotopy classes, then that knot type is said to be {\em Legendrian simple}.  Previous examples of Legendrian simple knot types include the unknot \cite{[EF]}, as well as torus knots and the figure eight knot \cite{[EH2]}.

{\em Iterated torus knots}, as topological knot types, can be defined recursively.  Let 1-iterated torus knots be simply torus knots $(P_1,q_1)$ with $P_1$ and $q_1$ co-prime nonzero integers, and $|P_1|, q_1 > 1$.  Here $P_1$ is the algebraic intersection with a longitude, and $q_1$ is the algebraic intersection with a meridian in the preferred framing for a torus representing the unknot.  Then for each $(P_1,q_1)$ torus knot, take a solid torus regular neighborhood $N((P_1,q_1))$; the boundary of this is a torus, and given a framing we can describe simple closed curves on that torus as co-prime pairs $(P_2,q_2)$, with $q_2 > 1$.  In this way we obtain all 2-iterated torus knots, which we represent as ordered pairs, $((P_1,q_1),(P_2,q_2))$.  Recursively, suppose the $(r-1)$-iterated torus knots are defined; we can then take regular neighborhoods of all of these, choose a framing, and form the $r$-iterated torus knots as ordered $r$-tuples $((P_1,q_1),...,(P_{r-1},q_{r-1}),(P_r,q_r))$, again with $P_r$ and $q_r$ co-prime, and $q_r > 1$.

For ease of notation, if we are looking at a general $r$-iterated torus knot type, we will refer to it as $K_r$; a Legendrian representative will usually be written as $L_r$.  Note that we will use the letter $r$ both for the rotation number and as an index for our iterated torus knots; context will distinguish between the two uses.

We will study iterated torus knots using two framings.  The first is the standard framing for a torus, where the meridian bounds a disc inside the solid torus, and we use the preferred longitude which bounds a Seifert surface in the complement of the solid torus.  We will refer to this framing as $\mathcal{C}$.  The second framing is a non-standard framing using a different longitude that comes from the cabling torus.  More precisely, to identify this non-standard longitude on $\partial N(K_r)$, we first look at $K_r$ as it is embedded in $\partial N(K_{r-1})$.  We take a small neighborhood $N(K_r)$ such that $\partial N(K_r)$ intersects $\partial N(K_{r-1})$ in two parallel simple closed curves.  These curves are longitudes on $\partial N(K_r)$ in this second framing, which we will refer to as $\mathcal{C'}$.  Note that this corresponds to the $\mathcal{C'}$ framing in \cite{[EH1]}, and is well-defined for any cabled knot type.  Moreover, for purpose of calculations there is an easy way to change between the two framings, which is presented in \cite{[EH1]} and which we will review in the body of this note.

Given a simple closed curve $(\mu, \lambda)$ on a torus, measured in some framing as having $\mu$ meridians and $\lambda$ longitudes, we will say this curve has slope of $\lambda/\mu$; i.e., longitudes over meridians.  Therefore we will refer to the longitude in the $\mathcal{C'}$ framing as $\infty^\prime$, and the longitude in the $\mathcal{C}$ framing as $\infty$.  The meridian in both framings will have slope $0$.  This way of representing slopes corresponds to that in \cite{[EH1]}; in short, slopes are the reciprocals of cabling fractions $\mu/\lambda$.

A new convention we will be using is that meridians in the standard $\mathcal{C}$ framing, that is, algebraic intersection with $\infty$, will be denoted by upper-case $P$.  On the other hand, meridians in the non-standard $\mathcal{C}'$ framing, that is, algebraic intersection with $\infty'$, will be denoted by lower-case $p$.

Given an iterated torus knot type $K_r = ((p_1,q_1),...,(p_r,q_r))$ where the $p_i$'s are measured in the $\mathcal{C}'$ framing, we define two quantities, whose meaning will be revealed in the body of this note.  The two quantities are:

\begin{equation}
\label{ArBr}
\displaystyle A_r := \sum_{\alpha=1}^{r}p_\alpha \prod_{\beta=\alpha+1}^{r}q_\beta \prod_{\beta=\alpha}^{r}q_\beta \ \ \ \ \ \ \ \ B_r := \sum_{\alpha=1}^r \left(p_\alpha \prod_{\beta=\alpha+1}^r q_\beta \right) + \prod_{\alpha=1}^r q_\alpha
\end{equation}

Note here we use a convention that $\displaystyle\prod_{\beta=r+1}^{r}q_\beta := 1$.  Also, if we restrict to the first $i$ iterations, that is, to $K_i = ((p_1,q_1),...,(p_i,q_i))$, we have an associated $A_i$ and $B_i$.  For example, $\displaystyle A_i := \sum_{\alpha=1}^{i}p_\alpha \prod_{\beta=\alpha+1}^{i}q_\beta \prod_{\beta=\alpha}^{i}q_\beta$.

Finally, for convenience in stating our theorems, we will define a particular class of iterated torus knot types, each member of which we will denote by $\breve{K}_r$:

\begin{definition}
{\em $\breve{K}_r = ((p_1,q_1),..., (p_i,q_i),...,(p_r,q_r))$ is an $r$-iterated torus knot type, where we require that $r \geq 1$, $q_i > 1$ for all $i$, $p_1 > 1$, and for $i \geq 1$ we have $q_{i+1}/p_{i+1} \notin (-1/B_i,0)$; at each iteration we use the $\mathcal{C'}$ framing.}
\end{definition}

We will show that the following is an equivalent definition for $\breve{K}_r$ in the $\mathcal{C}$ framing:  form an iterated torus knot by beginning with a positive $(P_1,q_1)$ torus knot, and then at each iteration take cabling fractions $P_{i+1}/q_{i+1}$ greater than $w(K_i)$.  Note also that for $\breve{K}_r$ we will show that $A_r > B_r > 0$.
 
We can now state our remaining results; our first is that the $\breve{K}_r$ are Legendrian simple:

\begin{theorem}
\label{main theorem}
Each $\breve{K}_r$ is Legendrian simple, and has a Legendrian mountain range with a single peak at $\overline{tb}=A_r-B_r=-\chi(\breve{K}_r)$ and $r=0$.
\end{theorem}

The Legendrian classification of the $\breve{K}_r$ generalizes that of positive torus knots, as their Legendrian mountain ranges are vertical translates of those for positive torus knots.  A result of Etnyre and Honda is that the $(2,3)$ torus knot fails the UTP; hence many of the $\breve{K}_r$ are iterated cablings that begin with knots failing the UTP.

%The Legendrian mountain range for the $((2,3),(5,3))$ iterated torus knot is shown in Figure \ref{fig:legmountainrangeB}.

%\begin{figure}[htbp]
%	\centering
%		\includegraphics[width=0.55\textwidth]{legmountainrangeB}
%	\caption{The Legendrian mountain range for the $((2,3),(5,3))$ iterated torus knot.  Arrows down and to the left represent negative stabilization; arrows down and to the right represent positive stabilization.}
%	\label{fig:legmountainrangeB}
%\end{figure}

We then determine more cablings of these $\breve{K}_r$ that are also Legendrian simple, and futhermore satisfy the UTP:

\begin{theorem}
\label{main theorem 3}
Let $K_{r+1}$ be a $(p_{r+1},q_{r+1})$ cabling of $\breve{K}_r$, where $q_{r+1}/p_{r+1} \in (-1/A_r,0)$, as measured in the $\mathcal{C'}$ framing.  Then $K_{r+1}$ is Legendrian simple, $\overline{tb}=A_{r+1}$, and the Legendrian mountain range can be determined based on the Legendrian classification of $\breve{K}_r$.  Moreover, $K_{r+1}$ satisfies the UTP.
\end{theorem}

Note that by Theorem \ref{UTP theorem}, all iterated cablings beginning with these $K_{r+1}$ are Legendrian simple.

Taken together, these two theorems show that all cablings of $\breve{K}_r$ with slopes in the complement of the interval $[-1/B_r,-1/A_r]$ are Legendrian simple.  This is not by accident; it will be shown that the slopes of dividing curves on the boundary of solid tori representing $\breve{K}_r$ that may fail to thicken will be contained within the interval $[-1/B_r,-1/A_r)$.

We will prove these two theorems using the $\mathcal{C}'$ framing, as they are stated.  However, after changing from $\mathcal{C}'$ to $\mathcal{C}$ via a change of coordinates, Theorems \ref{main theorem} and \ref{main theorem 3} will immediately imply the following Corollaries \ref{fractionslargerthanwidth} and \ref{negativecablings}, respectively, both of which are stated in the $\mathcal{C}$ framing:

\begin{corollary}
\label{fractionslargerthanwidth}
Let $K_r = ((P_1,q_1),...,(P_i,q_i),...,(P_r,q_r))$ be an iterated torus knot where $(P_1,q_1)$ is a positive torus knot, and such that $P_{i+1}/q_{i+1} > w(K_i) = \overline{tb}(K_i)$ for $1 \leq i < r$.  Then $K_r$ is Legendrian simple and has a Legendrian mountain range with a single peak at $\overline{tb}(K_r) = -\chi(K_r)$ and $r=0$.
\end{corollary}

\begin{corollary}
\label{negativecablings}
Let $K_{r+1}= ((P_1,q_1),...,(P_i,q_i),...,(P_r,q_r), (P_{r+1},q_{r+1}))$ be an iterated torus knot such that $P_{r+1} < 0$, $(P_1,q_1)$ is a positive torus knot and $P_{i+1}/q_{i+1} > w(K_i)$ for $1 \leq i < r$.  Then $K_{r+1}$ is Legendrian simple and satisfies the UTP.
\end{corollary}

%Note that these cablings will be shown to be {\em negative} in the standard $\mathcal{C}$ framing, and also have cabling fractions less than the value of $w(\breve{K}_r)$.  Also note that by Theorem \ref{UTP theorem}, all iterated cablings beginning with these $K_{r+1}$ are Legendrian simple.

Also note that the above classes of knots are transversally simple, since Legendrian simplicity implies transversal simplicity (see Theorem 2.10 in \cite{[EH2]}).

Figure \ref{fig:SchematicNewestB2} is a schematic indicating what is known and what is unknown about the uniform thickness and the Legendrian simplicity of iterated torus knots.  What is known is boxed; what is unknown is in bold with question marks.

\begin{figure}[htbp]
	\centering
		\includegraphics[width=.80\textwidth]{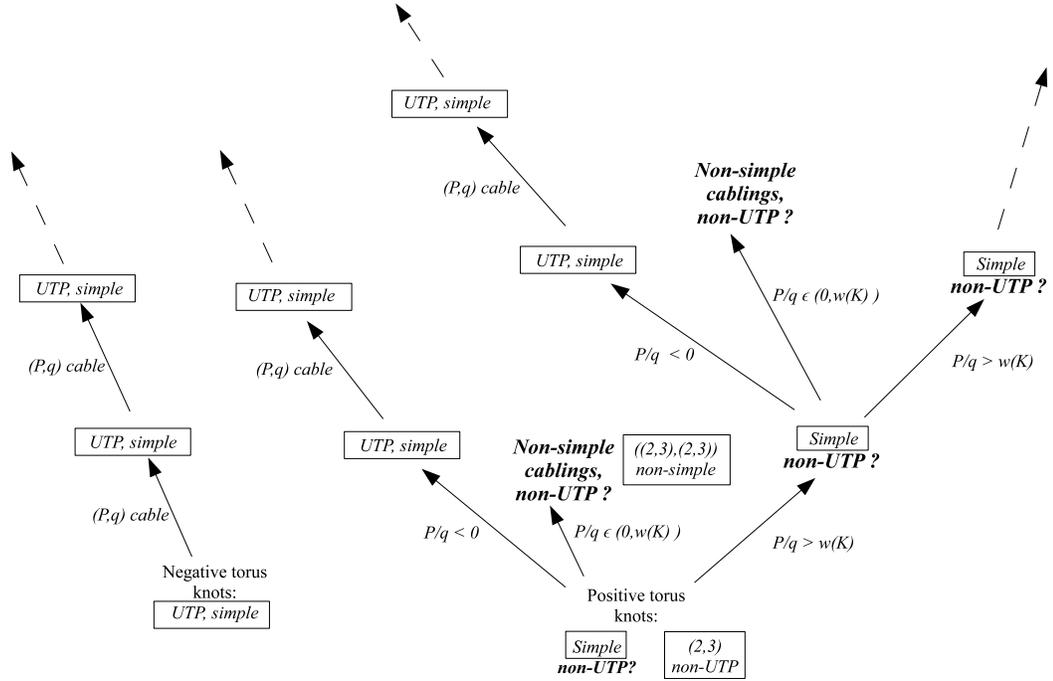}
	\caption{\small{Shown is a schematic that indicates what is known and unknown about uniform thickness and Legendrian simplicity of iterated torus knots.  What is known is boxed; what is unknown is in bold with question marks.  Each arrow represents a single cabling iteration in the standard $\mathcal{C}$ framing.}}
	\label{fig:SchematicNewestB2}
\end{figure}

Combining Theorem \ref{UTP theorem}, Corollaries \ref{fractionslargerthanwidth} and \ref{negativecablings}, and the fact that negative torus knots are simple and satisfy the UTP, yields the following necessary conditions for failure of the UTP for iterated torus knots.

\begin{corollary}
\label{necessarycondforUTP}
Suppose $K_r$ is an iterated torus knot type that fails the UTP. Then either:

\begin{itemize}
	\item[1.] $K_r = \breve{K}_r$.
	\item[2.] $K_r = ((P_1,q_1),..., (P_i,q_i),...,(P_r,q_r))$, where for some $1 \leq i < r$ we have $K_i = \breve{K}_i$ and $P_{i+1}/q_{i+1} \in (0,w(K_i))$.
\end{itemize}

\end{corollary}

Finally, again combining Theorem \ref{UTP theorem}, Corollaries \ref{fractionslargerthanwidth} and \ref{negativecablings}, and the fact that negative torus knots are simple and satisfy the UTP, we obtain the following necessary conditions for Legendrian non-simplicity of iterated torus knots:

\begin{corollary}
\label{necessarycondfornonsimple}
Suppose $K_r$ is an iterated torus knot type that is Legendrian non-simple.  Then $K_r = ((P_1,q_1),..., (P_i,q_i),...,(P_r,q_r))$, where for some $1 \leq i < r$ we have $K_i = \breve{K}_i$ and $P_{i+1}/q_{i+1} \in (0,w(K_i))$.
\end{corollary}

We will be using tools developed by Giroux, Kanda, and Honda, and used by Etnyre and Honda in their work, namely convex tori and annuli, the classification of tight contact structures on solid tori and thickened tori, and the Legendrian classificaton of torus knots.  Most of the results we use can be found in \cite{[H]}, \cite{[EH1]}, or \cite{[EH2]}, and if we use a lemma, proposition, or theorem from one of these works, it will be specifically referenced.  We will also briefly make use of facts involving the classical invariant for transversal isotopy classes, namely the self-linking number, $sl$.

With this in mind, this note will proceed as follows.  In \S 2 we prove Theorem \ref{UTP theorem}.  In \S 3 we perform preliminary calculations that allow us to outline a strategy for proving Theorem \ref{main theorem}.  This leads us to \S 4, where we examine solid tori representing $\breve{K}_r$, obtaining necessary conditions for those that fail to thicken, as well as calculating $w(\breve{K}_r)$.  In \S 5 we prove Theorem \ref{main theorem}, and in \S 6 we prove Theorem \ref{main theorem 3}.

\subsection*{Acknowledgements}

This work composes part of my PhD thesis at the University at Buffalo under the advisement of William Menasco, whom I wish to thank for many helpful discussions and suggestions.

\section{Cabling preserves simplicity and the UTP}
We first review some facts about Legendrian knots on convex tori.  Recall that the characteristic foliation induced by the contact structure on a convex torus can be assumed to have a standard form, where there are $2n$ parallel {\em Legendrian divides} and a one-parameter family of {\em Legendrian rulings}.  Parallel push-offs of the Legendrian divides gives a family of $2n$ {\em dividing curves}, referred to as $\Gamma$.  For a particular convex torus, the slope of components of $\Gamma$ is fixed and is called the {\em boundary slope} of any solid torus which it bounds; however, the Legendrian rulings can take on any slope other than that of the dividing curves by Giroux's Flexibility Theorem \cite{[G]}.  A {\em standard neighborhood} of a Legendrian knot $L$ will have two dividing curves and a boundary slope of $1/tb(L)$.

For a topological knot type $K$, if $N$ is a solid torus having a representative of $K$ as its core and convex boundary, then $N$ {\em fails to thicken} if for all $N' \supset N$, we have $\textrm{slope}(\Gamma_{\partial N'}) = \textrm{slope}(\Gamma_{\partial N})$.

Given a ruling curve $L = (P,q)$ on a convex torus $\partial N(K)$, then recall that section 2.1 in \cite{[EH1]} provides a relationship between the framings $\mathcal{C}'$ and $\mathcal{C}$ on $\partial N(L)$.  In terms of a change of basis, we can represent slopes $\lambda/\mu$ as column vectors and then get from a slope $\lambda/\mu'$, measured in $\mathcal{C}'$ on $\partial N(L)$, to a slope $\lambda/\mu$, measured in $\mathcal{C}$, by:

\begin{equation}
\displaystyle \left(\begin{array}{cc}
														1 & Pq\\
														0 & 1\end{array}\right)\left(\begin{array}{c}
														\mu '\\
														\lambda\end{array}\right) = \left(\begin{array}{c}
														\mu\\
														\lambda\end{array}\right)\nonumber
										\end{equation}

In other words, $\mu = \mu ' + Pq\lambda$.  If we then define $t$ to be the twisting of the contact planes along $L$ with respect to the $\mathcal{C}'$ framing on $\partial N(L)$, equation 2.1 in \cite{[EH1]} gives us:  
														
\begin{equation}
\label{5}
tb(L) = Pq + t(L)
\end{equation}

Observe that $t(L)$ is also the twisting of the contact planes with respect to the framing given by $\partial N$, and so is equal to $-1/2$ times the geometric intersection number of $L$ with $\Gamma_{\partial N}$.  The maximal twisting number with respect to this framing will be denoted by $\overline{t}$.

Finally, recall that if $\mathcal{A}$ is a convex annulus with Legendrian boundary components, then dividing curves are arcs with endpoints on either one or both of the boundary components; an annulus with no boundary-parallel dividing curves is said to be {\em standard convex}.

We can now prove Theorem \ref{UTP theorem}:

\begin{proof}
Recall that we have a knot $K$ that is Legendrian simple and satisfies the UTP.  By Theorem 1.3 in \cite{[EH1]}, we know that $(P,q)$ cables are simple and satisfy the UTP, provided $P/q < w(K)$.  Thus we only need to look at the case where $P/q > w(K)$.  We will refer to the $(P,q)$ cable as $K_{(P,q)}$.  From Theorem 3.2 in \cite{[EH1]}, we know that $K_{(P,q)}$ is Legendrian simple and that $\overline{t}(K_{(P,q)}) < 0$.  Moreover, we know from the same theorem that $K_{(P,q)}$ achieves $\overline{tb}(K_{(P,q)})$ as a Legendrian ruling curve on a convex torus with boundary slope $1/w(K)$ and two dividing curves.

To prove that $K_{(P,q)}$ satisfies the UTP, it suffices to show that any solid torus $N_{(P,q)}$ representing $K_{(P,q)}$ thickens to a standard neighborhood of a Legendrian knot at $\overline{tb}(K_{(P,q)})$.  So given a solid torus $N_{(P,q)}$, let $\mathcal{A}$ be a convex annulus connecting $\partial N_{(P,q)}$ to itself, with $\partial \mathcal{A}$ being two $\infty^\prime$ rulings so that $\partial N_{(P,q)} \backslash \partial \mathcal{A}$ consists of two annuli, one of which, along with $\mathcal{A}$, bounds a solid torus $\widehat{N}$ representing $K$ with $\widehat{N} \supset N_{(P,q)}$.  Now since $K$ satisfies the UTP, $\widehat{N}$ can be thickened to a standard neighborhood of a Legendrian knot at $\overline{tb}(K)$, which we call $N$.  See part (a) in Figure \ref{fig:UTP1BB}.
%Then let $\mathcal{A}_{(P,q)}$ be a convex annulus joining $\partial N$ to $N_{(P,q)}$ inside $N \backslash N_{(P,q)}$, so that $\partial \mathcal{A}_{(P,q)}$ are Legendrian representatives of $K_{(P,q)}$.  Then the $\partial N$-edge of $\mathcal{A}_{(P,q)}$ is at $\overline{tb}(K_{(P,q)})$, so we can thicken $N_{(P,q)}$ inside $N$ so that $t(\infty')=\overline{t}(K_{(P,q)})$.  This gives us that the boundary slope of this thickened $N_{(P,q)}$ is $\frac{a}{\overline{t}(K_{(P,q)})}$ in the $\mathcal{C}'$ framing.  So we need to show that $a=1$.

\begin{figure}[htbp]
	\centering
		\includegraphics[width=0.80\textwidth]{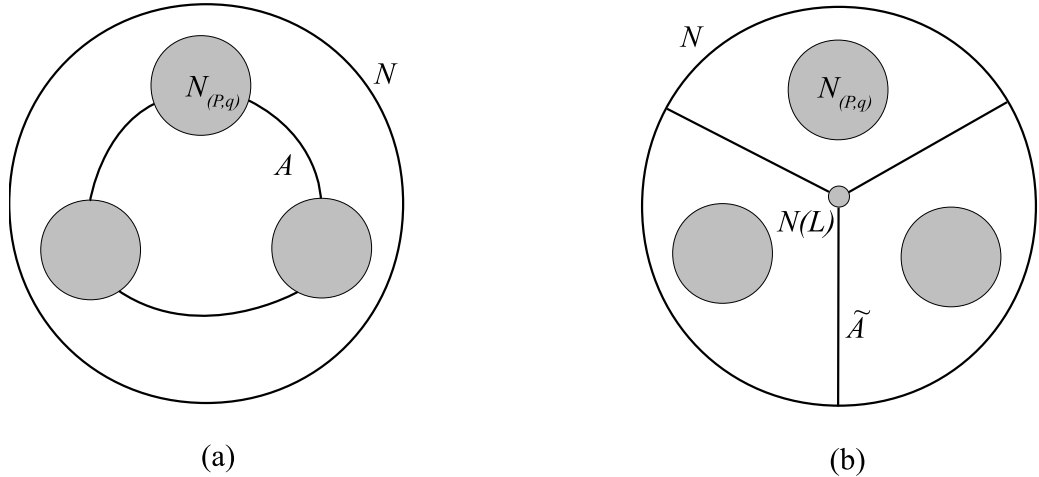}
	\caption{\small{Shown is a meridional cross-section of $N$.  The larger torus in gray is $N_{(P,q)}$; the smaller torus in gray is $N(L)$.}}
	\label{fig:UTP1BB}
\end{figure}

We now let $L$ be a Legendrian core curve representing $K$ in $\widehat{N} \backslash N_{(P,q)}$, and let $\widetilde{\mathcal{A}}$ be a convex annulus joining $\partial N$ to $\partial N(L)$ inside $N \backslash N_{(P,q)}$, with boundary components $(P,q)$ Legendrian rulings.  See part (b) in Figure \ref{fig:UTP1BB}.  We may assume that we have topologically isotoped $L$ so that the Thurston-Bennequin number is maximized over all such topological isotopies for the space $N \backslash N_{(P,q)}$.  $N(L)$ will have dividing curves of slope $1/m$ in $\mathcal{C}$, where $m\in \mathbb{Z}$.  We claim that in fact $m=\overline{tb}(K)$.  For if $m < \overline{tb}(K)$, then by the Imbalance Principle, there must exist bypasses on the $\partial N(L)$-edge of $\widetilde{\mathcal{A}}$, since the $\partial N$-edge of $\widetilde{\mathcal{A}}$ is at maximal twisting (see Prop 3.17 in \cite{[H]}).  But such a bypass would induce a destabilization of $L$, thus increasing its $tb$ by one -- see Lemma 4.4 in \cite{[H]}.  To satisfy the conditions of this lemma, we are using the fact that $P/q > w(K)$.  Thus $m=\overline{tb}(K)$ and $\widetilde{\mathcal{A}}$ is standard convex.

Finally, note that now $N_{(P,q)}$ thickens to $\widetilde{N}_{(P,q)}=N \backslash (N(\widetilde{\mathcal{A}}) \cup N(L))$.  We can calculate the boundary slope of $\widetilde{N}_{(P,q)}$.  We choose $(P',q')$ to be a curve on $N$ and $N(L)$ such that $Pq'-P'q=1$, and we change coordinates to a basis $\mathcal{C}''$ via the map $((P,q),(P',q')) \mapsto ((0,1),(-1,0))$.  Under this map we obtain 
\begin{equation}
\textrm{slope}(\Gamma_{\partial N}) = \textrm{slope}(\Gamma_{\partial N(L)})= \frac{q'w(K)-P'}{qw(K)-P}
\end{equation}

We then obtain in the $\mathcal{C}'$ framing, after edge-rounding, that

\begin{eqnarray}
\textrm{slope}(\Gamma_{\partial \widetilde{N}_{(P,q)}})&=&\frac{q'w(K)-P'}{qw(K)-P} -\frac{q'w(K)-P'}{qw(K)-P}+\frac{1}{qw(K)-P}\nonumber\\
																			&=&\frac{1}{qw(K)-P} = \frac{1}{\overline{t}(K_{(P,q)})} 
\end{eqnarray}

Hence the boundary slope of $\widetilde{N}_{(P,q)}$ must be $1/\overline{tb}(K_{(P,q)})$ with two dividing curves in the standard $\mathcal{C}$ framing.  Thus $K_{(P,q)}$ satisfies the UTP.
\end{proof}

\section{Preliminary calculations}

In this section we collect some identities and lemmas that will be useful in our analysis of iterated cablings that begin with positive torus knots.

First suppose $K_r = ((p_1,q_1),...,(p_r,q_r))$ is a general $r$-iterated torus knot type, with $p_i$'s measured in the $\mathcal{C}'$ framing.  We first obtain a formula for the $P_i$'s as measured in the standard $\mathcal{C}$ framing.  To this end, from equation \ref{ArBr} we obtain two useful identities:

\begin{equation}
\label{ArBrrecursive}
A_r = q_r^2 A_{r-1} + p_rq_r \qquad \qquad B_r = q_r B_{r-1} + p_r
\end{equation}

Now suppose we have a $((p_1,q_1),...,(p_r,q_r))$ iterated torus knot as described above, and let $P_i$ be the meridians for the $i$-th iteration, but as measured in the standard $\mathcal{C}$ framing.  To determine $P_{i+1}$, the algebraic intersection with the preferred longitude, we use the change of basis mentioned in \S 2 to obtain $P_{i+1} = q_{i+1}P_iq_i + p_{i+1}$.  We then can prove the following lemma:

\begin{lemma}
\label{Pr}
$P_r=q_rA_{r-1} + p_r$ for $r \geq 2$ and $A_r = P_r q_r$ for $r \geq 1$.
\end{lemma}

\begin{proof}
First observe that $P_1=p_1$ and so equation \ref{ArBr} immediately gives us $A_1=P_1q_1$.  We then use induction, beginning with a base case of $r=2$.  From the comments above we have $P_2=q_2A_1 + p_2$, and thus $A_2=P_2q_2$.  But then inductively we can assume that $A_{r-1}=P_{r-1}q_{r-1}$, and so again by the above comments $P_r=q_rA_{r-1}+p_r$, and hence $A_r = P_rq_r$.
\end{proof}

Note that as a consequence of this lemma, the change of coordinates from the $\mathcal{C}'$ framing to the $\mathcal{C}$ framing on $\partial N(K_r)$ becomes left multiplication by $\displaystyle \left(\begin{array}{cc}
														1 & A_r\\
														0 & 1\end{array}\right)$.

We now focus in on those particular iterated torus knot types $\breve{K}_r$ with $r \geq 1$, $q_i > 1$ for all $i$, $p_1 > 1$, and where for $i \geq 1$ we have $q_{i+1}/p_{i+1} \notin (-1/B_i,0)$.  We first prove a preliminary lemma concerning $A_r$, $B_r$, and $P_r$.

\begin{lemma}
$A_r > B_r > 0$ and $P_r > 0$ for any iterated torus knot type $\breve{K}_r$.
\end{lemma}

\begin{proof}
Observe that since $A_1 > B_1 > 0$ and $P_1=p_1 > 0$ for positive torus knots, we can assume inductively that $A_{r-1} > B_{r-1} > 0$ and that $P_{r-1} > 0$.  Then if $p_r > 0$, we certainly have $P_r > 0$ by Lemma \ref{Pr}; moreover, $A_r = q_r^2 A_{r-1} + p_rq_r > q_r A_{r-1} + p_r > q_rB_{r-1} + p_r = B_r > 0$.  In the other case, if $q_r/p_r < -1/B_{r-1}$, that means that $q_rB_{r-1}+p_r =B_r > 0 $.  Moreover, $P_r = q_rA_{r-1}+p_r > q_rB_{r-1}+p_r > 0$.  Finally, note that the previous proof that $A_r > B_r$ works for this case too.
\end{proof}

Recall that Lemma 2.2 in \cite{[EH1]} provides us with a way of calculating $r(L_r)$ from $r(\partial D)$ and $r(\partial \Sigma)$, where $D$ is a convex meridional disc for $N_{r-1}$ and $\Sigma$ is a convex Seifert surface for the preferred longitude on $N_{r-1}$.  Specifically, we have the equation:

\begin{equation}
\label{6}
r(L_r) = P_r r(\partial D) + q_r r(\partial \Sigma)
\end{equation}

We now can prove the following lemma:

\begin{lemma}
\label{tblemma}
$\overline{sl}(\breve{K}_r)=\overline{tb}(\breve{K}_r)=A_r - B_r = -\chi(\breve{K}_r)$
\end{lemma}

\begin{proof}
%We first show that both $\overline{sl}(\breve{K}_r)\leq A_r - B_r$ and $\overline{tb}(\breve{K}_r)\leq A_r - B_r$ using the Bennequin inequality.  To do this, we need to compute $\chi(\breve{K}_r)$, a formula for which is given at the end of the proof of Corollary 3 in \cite{[BW]}; in the notation used in that paper, the formula is $\chi(K_r)=\prod_{i=1}^r p_i - \sum_{i=1}^r q_i(p_i-1)\prod_{j=i+1}^r p_j$, since in our case all the $e_i=1$ as we are cabling positively at each iteration.  However, note that our $(P_i,q_i)$ corresponds to $(q_i,p_i)$ in \cite{[BW]} for $i > 1$; with this in mind, and some simplifying, we obtain $\chi(\breve{K}_r) = - A_r + B_r$.  The Bennequin inequality then gives us $tb \leq A_r-B_r$, as well as $sl \leq A_r-B_r$.
We first show that $\chi(\breve{K}_r) = B_r - A_r$; as a consequence, from the Bennequin inequality we obtain $\overline{sl}(\breve{K}_r) \leq A_r - B_r$ and $\overline{tb}(\breve{K}_r) \leq A_r - B_r$.

To this end, we use a formula for $\chi(K_r)$ given at the end of the proof of Corollary 3 in \cite{[BW]}.  The notation used in that paper is that an iterated torus knot $K_r$ is given by a sequence $(e_1(p_1,q_1),e_2(p_2,q_2),\cdots,e_r(p_r,q_r))$ where $p_i,q_i > 0$, $e_i = \pm 1$ indicates the parity of the cabling (either positive or negative), $e_1(p_1,q_1)$ is a torus knot, and for $i > 1$ the $p_i$ represent (efficient) geometric intersection with a meridian, while the $q_i$ represent (efficient) geometric intersection with a preferred longitude.  Using this notation from \cite{[BW]}, therefore, the formula of interest is:

\begin{equation}
\displaystyle\chi(K_r) = \prod_{i=1}^r p_i - \sum_{i=1}^r e_iq_i(p_i-1)\prod_{j=i+1}^r p_j - \sum_{i=1}^r (1-e_i)q_i(p_i-1)\prod_{j=i+1}^r p_j\nonumber
\end{equation}

We need to translate this formula into our notation.  For our $\breve{K}_r$ we have $e_i = 1$, since we are cabling positively at each iteration; also, our $(P_i,q_i)$ corresponds to $(q_i,p_i)$ in \cite{[BW]} for $i > 1$.  Thus our formula for $\chi(\breve{K}_r)$ is:

\begin{equation}
\displaystyle \chi(\breve{K}_r) = P_1 \prod_{i=2}^r q_i - q_1(P_1 - 1)\prod_{i=2}^r q_i - \sum_{i=2}^r P_i(q_i-1)\prod_{j=i+1}^r q_j \nonumber
\end{equation}

To show that this is equal to $B_r - A_r$, we need to rewrite it in terms of our $p_i$'s.  To do this, we note that $P_1 = p_1$, and from Lemma 3.1 we have for $i \geq 2$ that $\displaystyle P_i = p_i + q_i \sum_{\alpha = 1}^{i-1} p_\alpha \prod_{\beta=\alpha+1}^{i-1}q_\beta\prod_{\beta=\alpha}^{i-1}q_\beta$.  Our equation then becomes:

\begin{eqnarray*}
\displaystyle \chi(\breve{K}_r) &=& p_1 \prod_{i=2}^r q_i - q_1(p_1 - 1)\prod_{i=2}^r q_i\\
									& &-\sum_{i=2}^r \left(p_i + q_i \sum_{\alpha = 1}^{i-1} p_\alpha \prod_{\beta=\alpha+1}^{i-1}q_\beta\prod_{\beta=\alpha}^{i-1}q_\beta\right)(q_i-1)\prod_{j=i+1}^r q_j \nonumber
\end{eqnarray*}

If we distribute a few times and collect terms with plus signs, we obtain

\begin{eqnarray*}
\displaystyle \chi(\breve{K}_r) &=& p_1 \prod_{i=2}^r q_i + \prod_{i=1}^r q_i + \sum_{i=2}^r p_i \prod_{j=i+1}^r q_j -p_1 \prod_{i=1}^r q_i - \sum_{i=2}^r p_i \prod_{j=i}^r q_j \\
																& & -\sum_{i=2}^r \left(q_i \sum_{\alpha = 1}^{i-1} p_\alpha \prod_{\beta=\alpha+1}^{i-1}q_\beta\prod_{\beta=\alpha}^{i-1}q_\beta\right)(q_i-1)\prod_{j=i+1}^r q_j\nonumber
\end{eqnarray*}

The top line of the equation with plus signs is $B_r$; for the terms with minus signs, for each $i$ we can collect terms of like $p_i$ and obtain:

\begin{eqnarray*}
\displaystyle\chi(\breve{K}_r) &=& B_r - \sum_{i=1}^r p_i \left(\prod_{j=i}^r q_j +\sum_{j=i+1}^r q_j \prod_{\beta=i+1}^{j-1}q_\beta \prod_{\beta=i}^{j-1}q_\beta (q_j-1)\prod_{\beta=j+1}^r q_\beta\right)\\
																&=& B_r - \sum_{i=1}^r p_i \left(\prod_{j=i}^r q_j \left[1+\sum_{j=i+1}^r (q_j-1)\prod_{\beta=i+1}^{j-1}q_\beta\right]\right) = B_r - A_r\nonumber
\end{eqnarray*}

where in the last line we have used that $\displaystyle \left[1+\sum_{j=i+1}^r (q_j-1)\prod_{\beta=i+1}^{j-1}q_\beta\right] = \prod_{j=i+1}^r q_j$ (along with a notational convention that $\displaystyle \sum_{r+1}^r = 0$).

Then inductively we can assume $\overline{tb}(\breve{K}_{r-1})=A_{r-1}-B_{r-1}$ and there is a representative at that $tb$ value having $r=0$, since this is true for positive torus knots \cite{[EH2]}.  Then look at the $(p_r,q_r)$ cabling on a standard neighborhood of that representative of $\breve{K}_{r-1}$ at $tb=A_{r-1}-B_{r-1}$ and $r=0$.  Then the longitude and meridian both have $r=0$, and the twisting of the cabling equals $-B_r$.  Thus there is a representative of $\breve{K}_r$ at $tb=A_r-B_r$ and $r=0$, and hence $\overline{tb}(\breve{K}_r)= A_r-B_r$.  Moreover, by taking a positive transverse push-off, this proves $\overline{sl}(\breve{K}_r)= A_r-B_r$.
\end{proof}

Now in the Legendrian mountain range for $\breve{K}_r$, the outer left slope contains all Legendrian isotopy classes whose positive transverse push-offs are at $\overline{sl}$.  By the proof above, this slope must intersect the $r=0$ axis at $\overline{tb}=A_r-B_r$.  Since the mountain range is symmetric about the $r=0$ axis, we thus have the following corollary:

\begin{corollary}
The Legendrian mountain range for $\breve{K}_r$ consists of isotopy classes contained in a single peak centered around the line $r=0$ and with height at $\overline{tb}=A_r-B_r$.
\end{corollary}

The following will thus suffice to prove that $\breve{K}_r$ is Legendrian simple:

\begin{itemize}
\item[1.]  Show that there is a unique Legendrian isotopy class at $\overline{tb}=A_r-B_r$.
\item[2.]  Show that if $tb(L_r) < A_r-B_r$, then $L_r$ Legendrian destabilizes.
\end{itemize}

Recall from the work of Etnyre and Honda that a convenient way to find destabilizations of Legendrian knots embedded in tori is to find bypasses attached to these tori.  These bypasses can be found on either the interior or exterior of the solid tori, but with possible restrictions due to the failure of the UTP.  Thus, before we can prove Theorem \ref{main theorem}, we must turn our attention to the thickening of solid tori.

\section{Necessary conditions for solid tori $\breve{N}_r$ that do not thicken}

We begin with two new definitions that will be useful in this section.

\begin{definition}
{\em Let $N$ be a solid torus with convex boundary in standard form, and with $\textrm{slope}(\Gamma_{\partial N})=a/b$ in some framing.  If $|2b|$ is the geometric intersection of the dividing set $\Gamma$ with a longitude ruling in that framing, then we will call $a/b$ the {\em intersection boundary slope}}.
\end{definition}

Note that when we have an intersection boundary slope $a/b$, then $2\textrm{gcd}(a,|b|)$ is the number of dividing curves.

\begin{definition}
{\em For $r \geq 1$ and nonnegative integer $k$, define $N_r^k$ to be any solid torus representing $\breve{K}_r$ with intersection boundary slope of $-(k+1)/(A_rk+B_r)$, as measured in the $\mathcal{C}'$ framing.  Also define the integer $n_r^k :=\textrm{gcd}((k+1),(A_rk+B_r))$.}
\end{definition}

Note that $N_r^k$ has $2n_r^k$ dividing curves.

We will show that any solid torus $N_r$ representing $\breve{K}_r$ can be thickened to an $N_r^k$ for some nonnegative integer $k$, and that any solid torus with the same boundary slope as $N_r^k$ which fails to thicken must have at least $2n_r^k$ dividing curves.  Another way of saying this is that every solid torus $N_r$ is contained in some $N_r^k$, and that if $N_r$ fails to thicken, then boundary slopes do not change in passing to the $N_r^k \supset N_r$, although the number of dividing curves may decrease.

Our analysis proceeds by induction, where the base case is positive torus knots.  The following lemma is proved for the $(2,3)$ torus knot in \cite{[EH1]}, and there it is noted that there is a corresponding lemma for a positive $(p,q)$ torus knot.  However, the calculation is not explicitly provided, so for completeness we prove the general lemma here.

\begin{lemma}
\label{basecasethickening}
Let $N$ be a solid torus with core $\breve{K}_1 = (p,q)$ where $p,q > 1$ and co-prime.  Then $N$ can be thickened to an $N_1^k$ for some nonnegative integer $k$.  Moreover, if $N$ fails to thicken, then it has the same boundary slope as some $N_1^k$, as well as at least $2n_1^k$ dividing curves.
\end{lemma}

\begin{proof}
We first construct the setting.  Let $T$ be a torus which bounds solid tori $V_1$ and $V_2$ on both sides in $S^3$, and which contains a $(p,q)$ torus knot $\breve{K}_1$.  We will think of $T=\partial V_1$ and $T=-\partial V_2$.  Let $F_i$ be the core unknots for $V_i$.  We know $\overline{tb}(\breve{K}_1)=pq-p-q$ (see \cite{[EH2]}); measured with respect to the coordinate system $\mathcal{C}'$, for either $i$, $\overline{t}(\breve{K}_1)=-p-q$.

Now let $L_i$, $i=1,2$, be a Legendrian representative of $F_i$ with $tb=-m_i$, where $m_i > 0$ (recall that $\overline{tb}(\textrm{unknot})=-1$).  If $N(L_i)$ is a regular neighborhood of $L_i$, then $\textrm{slope}(\Gamma_{\partial N(L_i)})=-1/m_i$ with respect to $\mathcal{C}_{F_i}$.  %We recast these slopes with respect to a new coordinate system $\mathcal{C}''$ which identifies $T\cong \mathbb{R}^2/\mathbb{Z}^2$, where $K$ (viewed as sitting on $T$) corresponds to $(0,1)$.

Consider an oriented basis $((p,q),(p',q'))$ for $T$, where $pq'-qp'=1$; we map this to $((0,1),(-1,0))$ in a new framing $\mathcal{C}''$.  This corresponds to the map $\displaystyle \Phi_1 = \left(\begin{array}{cc}
														q &-p\\
														q' &-p'\end{array}\right)$.  Then $\Phi_1$ maps $(-m_1,1) \mapsto (-qm_1-p, -q'm_1-p')$.  Since we are only interested in slopes, we write this as $(qm_1+p, q'm_1+p')$.

Similarly, we change from $\mathcal{C}_{F_2}$ to $\mathcal{C}''$.  The only thing we need to know here is that $(-m_2,1)$ maps to $(pm_2+q,p'm_2+q')$.

This concludes the construction of the setting; we can now prove the lemma.  Let $N$ be a solid torus representing $\breve{K}_1$.  Let $L_i$ be Legendrian representatives of $F_i$ which maximize $tb(L_i)$ in the complement of $N$, subject to the condition that $L_1 \sqcup L_2$ is isotopic to $F_1 \sqcup F_2$ in the complement of $N$.%View $S^3\backslash(N(L_1) \cup N(L_2) \cup N)$ as a Seifert fibered space over the twice punctured sphere, where the annuli which connect among $N(L_1)$, $N(L_2)$, and $N$ admit fibrations by the Seifert fibers (which are topologically isotopic to $K$).

Now suppose $qm_1+p \neq pm_2 + q$.  This would mean that the twisting of Legendrian ruling representatives of $\breve{K}_1$ on $\partial N(L_1)$ and $\partial N(L_2)$ would be unequal.  Then we could apply the Imbalance Principle (see Proposition 3.17 in \cite{[H]}) to a convex annulus $\mathcal{A}$ in $S^3 \backslash N$ between $\partial N(L_1)$ and $\partial N(L_2)$ to find a bypass along one of the $\partial N(L_i)$.  This bypass in turn gives rise to a thickening of $N(L_i)$, allowing the increase of $tb(L_i)$ by one (see Lemma 4.4 in \cite{[H]}).  Hence, eventually we arrive at $qm_1+p = pm_2 + q$ and a standard convex annulus $\mathcal{A}$.

Since $m_i > 0$, the smallest solution to $qm_1+p = pm_2 + q$ is $m_1=m_2=1$.  All the other positive integer solutions are therefore obtained by taking $m_1=pk+1$ and $m_2=qk+1$ with $k$ a nonnegative integer.  We can then compute the intersection boundary slope of the dividing curves on $\partial(N(L_1) \cup N(L_2) \cup \mathcal{A})$, measured with respect to $\mathcal{C}'$, after edge-rounding.  This will be the intersection boundary slope for $\widetilde{N} \supset N$.  We have:

\begin{equation}
-\frac{q'(pk+1)+p'}{pqk+p+q} + \frac{p'(qk+1)+q'}{pqk+p+q}-\frac{1}{pqk+p+q} = -\frac{k+1}{pqk+p+q}=-\frac{k+1}{A_1k+B_1}
\end{equation}

This shows that any $N$ thickens to some $N_1^k$, and if $N$ fails to thicken, then it has the same boundary slope as some $N_1^k$.  Suppose, for contradiction, that $N$ fails to thicken and has $2n$ dividing curves, where $n < n_1^k$.  Then using the construction above we know that outside of $N$ in $S^3$ are neighborhoods of the two Legendrian unknots $L_i$ with $\breve{K}_1$ rulings that intersect the dividing set on $\partial N(L_i)$ exactly $2(A_1k+B_1)$ number of times.  However, since $n < n_1^k$, the $\infty'$ rulings on $N$ intersect the dividing set less than $2(A_1k+B_1)$ number of times.  Thus by the Imbalance Principle there exists bypasses off of the $\breve{K}_1$ rulings on the $\partial N(L_i)$, and so the $L_i$ can destabilize in the complement of $N$ to smaller $k$-value, allowing for a slope-changing thickening of $N$.  This is a contradiction. 
\end{proof}

%Note that Etnyre and Honda indicate that their proof that the $(2,3)$ torus knot is not uniformly thick generalizes to all positive torus knots \cite{[EH1]}. %Also, by the construction in the above lemma, we know that $2n$ is the minimal number of dividing curves for a solid torus $N$ having boundary slope $-\frac{k+1}{pqk+p+q}$.

We now can prove the following general result by induction using the above lemma as our base case:

\begin{lemma}
\label{nonthickening inductive step}
Let $N_r$ be a solid torus representing $\breve{K}_r$, for $r \geq 1$.  Then $N_r$ can be thickened to an $N_r^k$ for some nonnegative integer $k$.  Moreover, if $N_r$ fails to thicken, then it has the same boundary slope as some $N_r^k$, as well as at least $2n_r^k$ dividing curves.
\end{lemma}

\begin{proof}
Inductively we can assume that the lemma is true for solid tori $N_{r-1}$ representing $\breve{K}_{r-1}$.  Let $N_r$ be a solid torus representing $\breve{K}_r$.  Let $L_{r-1}$ be a Legendrian representative of $\breve{K}_{r-1}$ in $S^3\backslash N_r$ and such that we can join $\partial N(L_{r-1})$ to $\partial N_r$ by a convex annulus $\mathcal{A}_{(p_r,q_r)}$ whose boundaries are $(p_r,q_r)$ and $\infty'$ rulings on $\partial N(L_{r-1})$ and $\partial N_r$, respectively.  Then topologically isotop $L_{r-1}$ in the complement of $N_r$ so that it maximizes $tb$ over all such isotopies; this will induce an ambient topological isotopy of $\mathcal{A}_{(p_r,q_r)}$, where we still can assume $\mathcal{A}_{(p_r,q_r)}$ is convex.  In the $\mathcal{C}'$ framing we will have $\textrm{slope}(\Gamma_{\partial N(L_{r-1})})=-1/m$ where $m > 0$, since $\overline{t}(\breve{K}_{r-1})=-B_{r-1} < 0$.  Now if $m=B_{r-1}$, then there will be no bypasses on the $\partial N(L_{r-1})$-edge of $\mathcal{A}_{(p_r,q_r)}$, since the $(p_r,q_r)$ ruling would be at maximal twisting.  On the other hand, if $m > B_{r-1}$, then there will still be no bypasses on the $\partial N(L_{r-1})$-edge of $\mathcal{A}_{(p_r,q_r)}$, since such a bypass would induce a destabilization of $L_{r-1}$, thus increasing its $tb$ by one -- see Lemma 4.4 in \cite{[H]}.  To satisfy the conditions of this lemma, we are using the fact that either $p_r > 0$ or $q_{r}/p_{r} < -1/B_{r-1}$.  Furthermore, we can thicken $N_r$ through any bypasses on the $\partial N_r$-edge, and thus assume $\mathcal{A}_{(p_r,q_r)}$ is standard convex.  See (a) in Figure \ref{fig:non-thickening1B}.

\begin{figure}[htbp]
\label{nonthickening1}
	\centering
		\includegraphics[width=.85\textwidth]{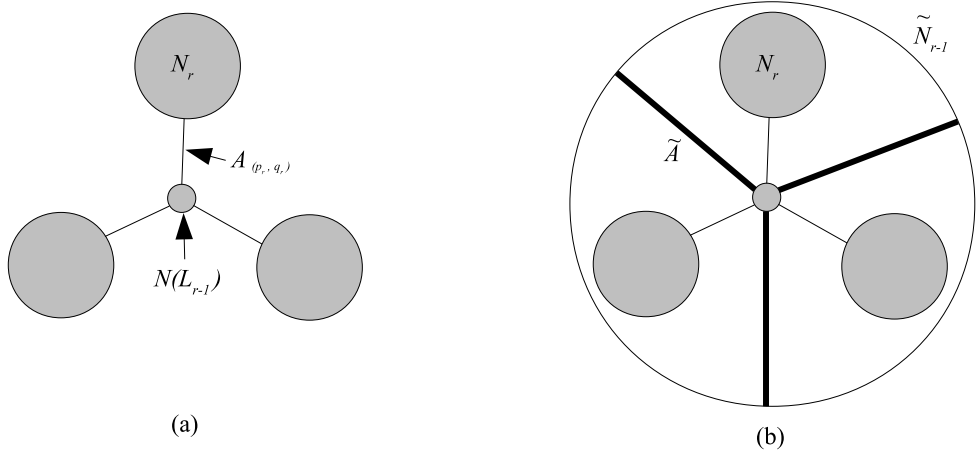}
	\caption{\small{$N_r$ is the larger solid torus in gray; $N(L_{r-1})$ is the smaller solid torus in gray.}}
	\label{fig:non-thickening1B}
\end{figure}

Now let $N_{r-1}:=N_r \cup N(\mathcal{A}_{(p_r,q_r)}) \cup N(L_{r-1})$.  By our inductive hypothesis we can thicken $N_{r-1}$ to an $\widetilde{N}_{r-1}$ with intersection boundary slope $-(k_{r-1} + 1)/(A_{r-1}k_{r-1} + B_{r-1})$, and we can assume that $k_{r-1}$ is minimized for all such thickenings.  Then consider a convex annulus $\mathcal{\widetilde{A}}$ from $\partial N(L_{r-1})$ to $\partial \widetilde{N}_{r-1}$, such that $\mathcal{\widetilde{A}}$ is in the complement of $N_r$ and $\partial \mathcal{\widetilde{A}}$ consists of $(p_r,q_r)$ rulings.  See (b) in Figure \ref{fig:non-thickening1B}.  We will show that $\mathcal{\widetilde{A}}$ is standard convex.  Certainly there are no bypasses on the $\partial N(L_{r-1})$-edge of $\mathcal{\widetilde{A}}$; furthermore, any bypasses on the $\partial \widetilde{N}_{r-1}$-edge must pair up via dividing curves on $\partial \widetilde{N}_{r-1}$ and cancel each other out as in part (a) of Figure \ref{nonthickening2}, for otherwise a bypass on $\partial N(L_{r-1})$ would be induced via the annulus $\mathcal{\widetilde{A}}$ as in part (b) of Figure \ref{nonthickening2}.  As a consequence, allowing $\widetilde{N}_{r-1}$ to thin inward through such bypasses does not change the boundary slope, but just reduces the number of dividing curves.  But then inductively we can thicken this new $\widetilde{N}_{r-1}$ to a smaller $k_{r-1}$-value, contradicting the minimality of $k_{r-1}$.  Thus $\mathcal{\widetilde{A}}$ is standard convex.

\begin{figure}[htbp]
\label{nonthickening2}
	\centering
		\includegraphics[width=0.80\textwidth]{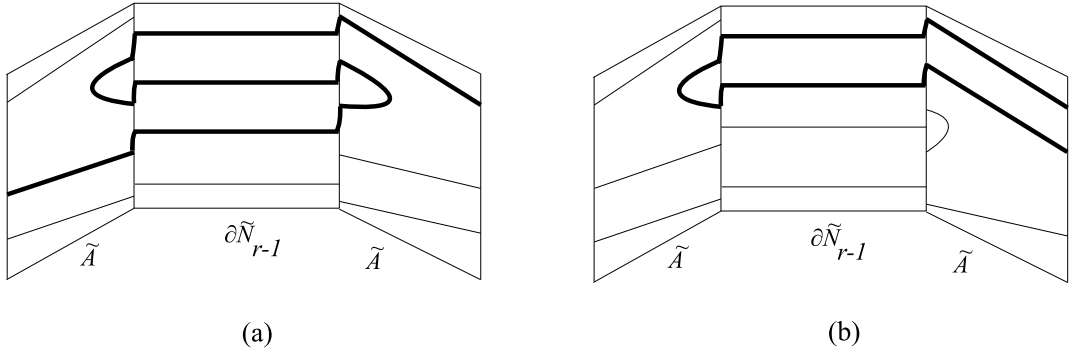}
	\caption{\small{Part (a) shows bypasses that cancel each other out after edge-rounding.  Part (b) shows a bypass induced on $\partial N(L_{r-1})$ via $\widetilde{\mathcal{A}}$.}}
	\label{fig:non-thickening2B}
\end{figure}

Now four annuli compose the boundary of a solid torus $\widetilde{N}_r$ containing $N_r$:  the two sides of a thickened $\mathcal{\widetilde{A}}$; $\partial \widetilde{N}_{r-1} \backslash \partial \mathcal{\widetilde{A}}$; and $\partial N(L_{r-1}) \backslash \partial \mathcal{\widetilde{A}}$.  We can compute the intersection boundary slope of this solid torus.  To this end, recall that $\textrm{slope}(\Gamma_{\partial N(L_{r-1})})=-1/m$ where $m > 0$.  To determine $m$ we note that the geometric intersection of $(p_r,q_r)$ with $\Gamma$ on $\partial \widetilde{N}_{r-1}$ and $\partial N(L_{r-1})$ must be equal, yielding the equality

\begin{equation}
p_r + mq_r = p_rk_{r-1} + p_r + q_r(A_{r-1}k_{r-1} + B_{r-1})
\end{equation}

This gives

\begin{equation}
m=p_r\frac{k_{r-1}}{q_r}+A_{r-1}k_{r-1} + B_{r-1}
\end{equation}

We define the integer $k_r:=k_{r-1}/q_r$.  We now choose $(p'_r,q'_r)$ to be a curve on these two tori such that $p_rq'_r-p'_rq_r=1$, and as in Lemma \ref{basecasethickening}, we change coordinates to $\mathcal{C}''$ via the map $((p_r,q_r),(p'_r,q'_r)) \mapsto ((0,1),(-1,0))$.  Under this map we obtain 
\begin{equation}
\textrm{slope}(\Gamma_{\partial \widetilde{N}_{r-1}}) = \frac{q'_r(A_{r-1}k_{r-1}+B_{r-1})+p'_r(q_rk_r+1)}{A_rk_r+B_r}
\end{equation}

\begin{equation}
\textrm{slope}(\Gamma_{\partial N(L_{r-1})})=\frac{q'_r(p_rk_r +A_{r-1}k_{r-1}+B_{r-1}) +p'_r}{A_rk_r+B_r}
\end{equation}

We then obtain in the $\mathcal{C}'$ framing, after edge-rounding, that the intersection boundary slope of $\widetilde{N}_r$ is

\begin{eqnarray}
\textrm{slope}(\Gamma_{\partial \widetilde{N}_r})&=&\frac{q'_r(A_{r-1}k_{r-1}+B_{r-1})+p'_r(q_rk_r+1)}{A_rk_r+B_r}\nonumber\\
																			&-&\frac{q'_r(p_rk_r +A_{r-1}k_{r-1}+B_{r-1}) +p'_r}{A_rk_r+B_r}\nonumber\\
																			&-&\frac{1}{A_rk_r+B_r}\nonumber\\
																			&=&-\frac{k_r+1}{A_rk_r+B_r}
\end{eqnarray}

This shows that any $N_r$ representing $\breve{K}_r$ can be thickened to one of the $N_r^k$, and if $N_r$ fails to thicken, then it has the same boundary slope as some $N_r^k$.  We now show that if $N_r$ fails to thicken, and if it has the minimum number of dividing curves over all such $N_r$ which fail to thicken and have the same boundary slope as $N_r^k$, then $N_r$ is actually an $N_r^k$.  

To see this, as above we can choose a Legendrian $L_{r-1}$ that maximizes $tb$ in the complement of $N_r$ and such that we can join $\partial N(L_{r-1})$ to $\partial N_r$ by a convex annulus $\mathcal{A}_{(p_r,q_r)}$ whose boundaries are $(p_r,q_r)$ and $\infty'$ rulings on $\partial N(L_{r-1})$ and $\partial N_r$, respectively.  Again we have no bypasses on the $\partial N(L_{r-1})$-edge, and in this case we have no bypasses on the $\partial N_r$-edge since $N_r$ fails to thicken and is at minimum number of dividing curves.

As above, let $N_{r-1}:=N_r \cup N(\mathcal{A}_{(p_r,q_r)}) \cup N(L_{r-1})$.  We claim this $N_{r-1}$ fails to thicken.  To see this, take a convex annulus $\mathcal{\widetilde{A}}$ from $\partial N(L_{r-1})$ to $\partial N_{r-1}$, such that $\mathcal{\widetilde{A}}$ is in the complement of $N_r$ and $\partial \mathcal{\widetilde{A}}$ consists of $(p_r,q_r)$ rulings.  We know $\mathcal{\widetilde{A}}$ is standard convex since the twisting is the same on both edges and there are no bypasses on the $\partial N(L_{r-1})$-edge.  A picture is shown in Figure \ref{fig:non-thickening3BB}.

\begin{figure}[htbp]
	\centering
		\includegraphics[width=0.40\textwidth]{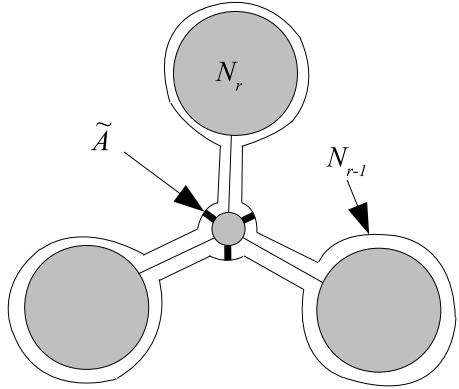}
	\caption{\small{Shown is a meridional cross-section of $N_{r-1}$.  The larger gray solid torus represents $N_r$; the smaller gray solid torus is $N(L_{r-1})$.}}
	\label{fig:non-thickening3BB}
\end{figure}

Now four annuli compose the boundary of a solid torus containing $N_r$:  the two sides of the thickened $\mathcal{\widetilde{A}}$, which we will call $\mathcal{\widetilde{A}}_+$ and $\mathcal{\widetilde{A}}_-$; $\partial N_{r-1} \backslash \partial \mathcal{\widetilde{A}}$, which we will call $\mathcal{A}_{r-1}$; and $\partial N(L_{r-1}) \backslash \partial \mathcal{\widetilde{A}}$, which we will call $\mathcal{A}_{L_{r-1}}$.  Any thickening of $N_{r-1}$ will induce a thickening of $N_r$ to $\widetilde{N}_r$ via these four annuli.

Suppose, for contradiction, that $N_{r-1}$ thickens outward so that $\textrm{slope}(\Gamma_{\partial N_{r-1}})$ changes.  Note that during the thickening, $\mathcal{A}_ {L_{r-1}}$ stays fixed.  We examine the rest of the annuli by breaking into two cases.

\textbf{Case 1:}  After thickening, suppose $\mathcal{\widetilde{A}}$ is still standard convex; that means both $\mathcal{\widetilde{A}}_+$ and $\mathcal{\widetilde{A}}_-$ are standard convex.  Since we can assume that after thickening $\mathcal{A}_{r-1}$ is still standard convex, this means that in order for $\textrm{slope}(\Gamma_{\partial N_{r-1}})$ to change, the holonomy of $\Gamma_{\mathcal{A}_{r-1}}$ must have changed.  But this will result in a change in $\textrm{slope}(\Gamma_{\partial N_r})$, since $\mathcal{A}_ {L_{r-1}}$ stays fixed and any change in holonomy of $\Gamma_{\mathcal{\widetilde{A}}_+}$ and $\Gamma_{\mathcal{\widetilde{A}}_-}$ cancels each other out and does not affect $\textrm{slope}(\Gamma_{\partial N_r})$.  Thus we would have a slope-changing thickening of $N_r$, which by hypothesis cannot occur.

\textbf{Case 2:}  After thickening, suppose $\mathcal{\widetilde{A}}$ is no longer standard convex.  Now note that there are no bypasses on the $\partial N(L_{r-1})$-edge of $\mathcal{\widetilde{A}}$; furthermore, any bypass for $\mathcal{\widetilde{A}}_+$ on the $\partial N_{r-1}$-edge must be cancelled out by a corresponding bypass for $\mathcal{\widetilde{A}}_-$ on the $\partial N_{r-1}$-edge as in part (a) of Figure \ref{fig:non-thickening2B}, so as not to induce a bypass on the $\partial N(L_{r-1})$-edge as in part (b) of the same figure.  But then again, in order for $\textrm{slope}(\Gamma_{\partial N_r})$ to remain constant, the holonomy of $\Gamma_{\mathcal{A}_{r-1}}$ must remain constant, and thus $\textrm{slope}(\Gamma_{\partial N_{r-1}})$ must also have remained constant, with just an increase in the number of dividing curves.

This proves the claim that $N_{r-1}$ does not thicken, and we therefore know that its boundary slope is $-(k_{r-1}+1)/(A_{r-1}k_{r-1}+B_{r-1})$.  Furthermore, we know the number of dividing curves is $2n$ where $n \geq n_{r-1}^{k_{r-1}}$.  Suppose, for contradiction, that $n > n_{r-1}^{k_{r-1}}$.  Then we know we can thicken $N_{r-1}$ to an $N_{r-1}^{k_{r-1}}$, and if we take a convex annulus from $\partial N_{r-1}$ to $\partial N_{r-1}^{k_{r-1}}$ whose boundaries are $(p_r,q_r)$ rulings, by the Imbalance Principle there must be bypasses on the $\partial N_{r-1}$-edge.  But these would induce bypasses off of $\infty'$ rulings on $N_r$, which by hypothesis cannot exist.  Thus $n = n_{r-1}^{k_{r-1}}$, and by a calculation as above we obtain that the intersection boundary slope of $N_r$ must be $-(k_{r}+1)/(A_{r}k_{r}+B_{r})$ for the integer $k_r=k_{r-1}/q_r$.
\end{proof}

Note the following inequality, which, among other things, shows that the boundary slopes of solid tori representing $\breve{K}_r$ that may fail to thicken are contained in the interval $[-1/B_r, -1/A_r)$.

\begin{equation}
\label{simple 1}
-\frac{1}{B_r} < -\frac{2}{A_r + B_r} < -\frac{3}{2A_r + B_r} < \cdots < -\frac{k_r + 1}{A_rk_r + B_r} < \cdots < -\frac{1}{A_r}
\end{equation}

To conclude this section, we have the following lemma:

\begin{lemma}
\label{widthKr}
$w(\breve{K}_r)=\overline{tb}(\breve{K}_r)$
\end{lemma}

\begin{proof}
Using the inequality above, it suffices to show that any solid torus $N_r$ representing $\breve{K}_r$ can be thickened to a solid torus with boundary slope $-(k_r + 1)/(A_rk_r + B_r)$ for some nonnegative integer $k_r$, for then to prevent overtwisting it would have to be the case that $\textrm{slope}(\Gamma_{\partial N_r}) \in [-1/B_r,0)$.  But by the above lemma this is true. 
\end{proof}

\section{Legendrian simplicity of $\breve{K}_r$}

We now use the strategy outlined in \S 3 to prove Theorem \ref{main theorem}.  Since Theorem \ref{main theorem} is true for positive torus knots \cite{[EH2]}, we can inductively assume that it holds for $\breve{K}_{r-1}$.  We then prove it true for $\breve{K}_r$.  The proof will parallel the proof from \cite{[EH1]} that $K$ being simple and satisfying the UTP guarantees simplicity of cablings for cabling fractions that are greater than the contact width.  However, in our case $\breve{K}_{r-1}$ may not satisfy the UTP, so we will need appropriate modifications for our proof. 

\begin{proof}
We begin by showing that if $L_r$ and $L_r^\prime$ have maximal $\overline{tb}(\breve{K}_r)=A_r-B_r$, then they are Legendrian isotopic.  Now $\overline{t}(\breve{K}_r)=-B_r < 0$, so we can assume that both $L_r$ and $L_r^\prime$ exist as Legendrian rulings on convex tori $\partial N_{r-1}$ and $\partial N_{r-1}^\prime$.  Let $\textrm{slope}(\Gamma_{\partial N_{r-1}})=-\frac{a}{b}$ be an intersection boundary slope where $a,b>0$.  Then $-a/b \geq -1/B_{r-1}$, and we have $b \geq aB_{r-1}$.  But since $t(L_r)=-B_r$, we also have $ap_r + bq_r = B_r$.  Combining this equality and inequality we obtain $B_r \geq ap_r + aq_rB_{r-1}=aB_r$, which implies $a=1$ and $b=B_{r-1}$.  Hence, we can assume that $L_r$ lies on a convex torus with boundary slope $-1/B_{r-1}$, and similarly for $L_{r}^\prime$.

Now by Proposition 4.3 in \cite{[H]}, each solid torus with boundary slope $-1/B_{r-1}$ is contact isotopic to the standard neighborhood of a Legendrian representative of $K_{r-1}$ with $t(L_{r-1})=-B_{r-1}$; both $L_{r}$ and $L_{r}^\prime$ are Legendrian rulings on such a boundary torus.  But inductively there is only one such Legendrian $L_{r-1}$ at maximal $\overline{t}(K_{r-1})=-B_{r-1}$.  Thus, as in the proof of Lemma 3.4 in \cite{[EH1]}, we may assume that $L_{r}$ and $L_{r}^\prime$ are Legendrian rulings on the same boundary torus, and hence Legendrian isotopic via the rulings.

We now show that if $tb(L_r) < \overline{tb}(\breve{K}_r)$ then $L_r$ destabilizes using a bypass.  To this end, we note that since $q_r > 1$, we have

\begin{equation}
\label{simple 2}
 -\frac{1}{(A_r/q_r)} < -\frac{2}{A_r + B_r}
\end{equation} 

We first suppose that $t(L_r) = -m$, where $B_r < m \leq (A_r/q_r)$ (note that $B_r < (A_r/q_r)$ for $r>1$).  Then $N(L_r)$ has boundary slope $-1/m \leq -1/(A_r/q_r)$, and this, combined with Lemma \ref{nonthickening inductive step} and inequalities \ref{simple 1} and \ref{simple 2}, allows us to conclude that $N(L_r)$ can be thickened to a solid torus $N_r$ with intersection boundary slope $-1/B_r$.  Then an $\infty^\prime$ ruling on $N(L_r)$ can be destabilized using a bypass on a convex annulus joining the two tori.

Now suppose alternatively that $m > (A_r/q_r)$.  In this case, we look at $L_r$ as a $(p_r,q_r)$ Legendrian ruling on the convex boundary of a solid torus $N_{r-1}$ with boundary slope $s$.  We may assume that $L_r$ intersects the dividing set efficiently, for otherwise $L_r$ immediately destabilizes.  Note first that if $L_r^\prime$ is a $(p_r,q_r)$ ruling on a solid torus with intersection boundary slope $-1/A_{r-1}$, then $t(L_r^\prime)=-(A_r/q_r)$.  In light of this, note that by Lemmas \ref{nonthickening inductive step} and \ref{widthKr} and inequality \ref{simple 1}, as well as Lemmas 3.15 and 3.16 in \cite{[EH2]}, we must have $N_{r-1}$ either containing a solid torus with intersection boundary slope $-1/A_{r-1}$ (if $s <  -1/A_{r-1}$), or $N_{r-1}$ must thicken to a solid torus of intersection boundary slope $-1/A_{r-1}$ (if $s > -1/A_{r-1}$).  Either way, we can connect $L_r$ to an $L_r^\prime$ via a convex annulus and destabilize $L_r$ using a bypass.
\end{proof}

This proves Theorem \ref{main theorem}; a change of coordinates from $\mathcal{C}'$ to $\mathcal{C}$ then yields Corollary \ref{fractionslargerthanwidth}.
%It thus remains to consider the case $t(L_r) = -(A_r/q_r)$.  We view $L_r$ as a $(p_r,q_r)$ Legendrian ruling on the convex boundary of a solid torus $N_{r-1}$ with boundary slope $s$.  Now if $s \neq -\frac{1}{A_{r-1}}$, then $N_{r-1}$ either contains a torus with boundary slope $-\frac{1}{A_{r-1}}$ or can be thickened to one.  A convex annulus joining $L_r$ to the torus with slope $-\frac{1}{A_{r-1}}$ will then either have a bypass on the $L_r$-edge or be a Legendrian isotopy.  Thus we can assume that $s = -\frac{1}{A_{r-1}}$.  Then again by Proposition 4.3 in \cite{[H]}, we may assume that $L_r$ is a $(p_r,q_r)$ Legendrian ruling on the convex boundary of $N(L_{r-1})$ where $t(L_{r-1})=-A_{r-1}$.  But now since $K_{r-1}$ is a Legendrian simple knot type, we can destabilize $L_{r-1}$ to have $\overline{t}(K_r)=-B_r$.  We can do this so that it induces a thickening of $N(L_{r-1})$ to a solid torus with boundary slope $-\frac{1}{B_{r-1}}$; $L_r$ will then destabilize via a bypass on a convex annulus joining the two tori.

\section{Legendrian simple cablings of $\breve{K}_r$ that satisfy the UTP}

We now prove Theorem \ref{main theorem 3}:

\begin{proof}
Recall that we are given $q_{r+1}/p_{r+1} \in (-1/A_r,0)$.  Note first that in this case $P_{r+1} = p_{r+1}+q_{r+1}A_r < 0$ in the $\mathcal{C}$ framing.  Moreover, since $w(\breve{K}_r)=A_r-B_r > 0$, we have that $P_{r+1}/q_{r+1} < w(\breve{K}_r)$.  Our proof for this case will thus parallel the proof in \cite{[EH1]} that $K$ being Legendrian simple and satisfying the UTP, along with $P/q < w(K)$, guarantees that the $(P,q)$ cabling is also Legendrian simple and satisfies the UTP.  In our case, $\breve{K}_r$ does not necessarily satisfy the UTP, and thus we will need appropriate modifications for our proof.

The proof will require five steps:

\begin{itemize}
	\item[1.] Show that $\overline{tb}(K_{r+1})=A_{r+1}$.
	\item[2.] Show that $K_{r+1}$ satisfies the UTP.
	\item[3.] Calculate $r(L_{r+1})$ at $\overline{tb}$ and show that Legendrian isotopy classes at $\overline{tb}$ are determined by their rotation numbers.
	\item[4.] Show that if $tb(L_{r+1}) < \overline{tb}$, then $L_{r+1}$ destabilizes.
	\item[5.] Show that if $L_{r+1}$ is in a valley of the Legendrian mountain range (ie, $(r(L_{r+1})\pm 1,tb(L_{r+1})+ 1)$ have images in the mountain range, but $(r(L_{r+1}),tb(L_{r+1})+2)$ does not), then $L_{r+1}$ can destabilize both positively and negatively. 
\end{itemize}

\textbf{Step 1:}  Our analysis in the first two steps will draw heavily from ideas in the proof of Theorem 1.2 in \cite{[EH1]} that negative torus knots satisfy the UTP.  We first examine representatives of $K_{r+1}$ at $\overline{tb}$.  Since there exists a convex torus representing $\breve{K}_r$ with Legendrian divides that are $(p_{r+1},q_{r+1})$ cablings (inside of the solid torus representing $\breve{K}_r$ with $\textrm{slope}(\Gamma)=-1/A_r$) we know that $\overline{tb}(K_{r+1}) \geq P_{r+1}q_{r+1}=A_{r+1}$.  To show that $\overline{tb}(K_{r+1}) = A_{r+1}$, we show that $\overline{t}(K_{r+1})=0$ by showing that the contact width $w(K_{r+1},\mathcal{C}')=0$, since this will yield $\overline{tb}(K_{r+1}) \leq w(K_{r+1})=A_{r+1}$.  So suppose, for contradiction, that some $N_{r+1}$ has convex boundary with $\textrm{slope}(\Gamma_{\partial N_{r+1}})=s > 0$, as measured in the $\mathcal{C}'$ framing, and two dividing curves.  After shrinking $N_{r+1}$ if necessary, we may assume that $s$ is a large positive integer.  Then let $\mathcal{A}$ be a convex annulus from $\partial N_{r+1}$ to itself having boundary curves with slope $\infty'$.  Taking a neighborhood of $N_{r+1} \cup \mathcal{A}$ yields a thickened torus $R$ with boundary tori $T_1$ and $T_2$, arranged so that $T_1$ is inside the solid torus $N_r$ representing $\breve{K}_r$ bounded by $T_2$.

Now there are no boundary parallel dividing curves on $\mathcal{A}$, for otherwise, we could pass through the bypass and increase $s$ to $\infty'$, yielding excessive twisting inside $N_{r+1}$.  Hence $\mathcal{A}$ is in standard form, and consists of two parallel nonseparating arcs.  We now choose a new framing $\mathcal{C}''$ for $N_r$ where $(p_{r+1},q_{r+1}) \mapsto (0,1)$; then choose $(p'',q'') \mapsto (1,0)$ so that $p''q_{r+1}-q''p_{r+1}=1$ and such that $\textrm{slope}(\Gamma_{T_1})=-s$ and $\textrm{slope}(\Gamma_{T_2})=1$.  As mentioned in the third paragraph of the proof of Theorem 1.2 in \cite{[EH1]}, this is possible since $\Gamma_{T_1}$ is obtained from $\Gamma_{T_2}$ by $s+1$ right-handed Dehn twists.  Then note that in the $\mathcal{C'}$ framing, we have that $q_{r+1}/p_{r+1} > \textrm{slope}(\Gamma_{T_2})=(q''+q_{r+1})/(p''+p_{r+1}) > q''/p''$, and $q_{r+1}/p_{r+1}$ and $q''/p''$ are connected by an arc in the Farey tessellation of the hyperbolic disc (see section 3.4.3 in \cite{[H]}).  Thus, since $-1/A_r$ is connected by an arc to $0/1$ in the Farey tessellation, we must have that $(q''+q_{r+1})/(p''+p_{r+1}) > -1/A_r$.  Thus we can thicken $N_r$ to a standard neighorhood with $\textrm{slope}(\Gamma)=-1/A_r$.  Then, just as in Claim 4.2 in \cite{[EH1]}, we have the following:

\begin{itemize}
\item[(i)] inside $R$ there exists a convex torus parallel to $T_i$ with slope $q_{r+1}/p_{r+1}$; 
\item[(ii)] $R$ can thus be decomposed into two layered {\em basic slices}; 
\item[(iii)] the tight contact structure on $R$ must have {\em mixing of sign} in the Poincar$\acute{\textrm{e}}$ duals of the relative half-Euler classes for the layered basic slices; 
\item[(iv)] this mixing of sign cannot happen inside the universally tight standard neighborhood with $\textrm{slope}(\Gamma)=-1/A_r$.
\end{itemize}

This contradicts $s > 0$.  So $\overline{tb}(K_{r+1})=P_{r+1}q_{r+1}=A_{r+1}$.

\textbf{Step 2:}  Here we show that any $N_{r+1}$ can be thickened to a standard neighborhood of $L_{r+1}$ with $t(L_{r+1})=0$.  So suppose that $N_{r+1}$ has convex boundary with $\textrm{slope}(\Gamma_{\partial N_{r+1}})=s$, as measured in the $\mathcal{C}'$ framing, where $-\infty' < s < 0$.  Construct $R$ as in Step 1 above, and look at the convex annulus $\mathcal{A}$, which in this case may not be standard convex.  If all dividing curves on $\mathcal{A}$ are boundary parallel arcs, then $N_{r+1}$ can be thickened to have boundary slope $\infty'$.  On the other hand, if there are nonseparating dividing curves on $\mathcal{A}$ after going through bypasses, then the resulting $T_2$ will have negative boundary slope in the $\mathcal{C}''$ framing, and we can thicken $N_r$ to obtain a convex torus outside of $R$ on the $T_2$-side with slope $q_{r+1}/p_{r+1}$ in the $\mathcal{C}'$ framing, since $q_{r+1}/p_{r+1} > -1/A_r$ and thickening can occur.  Then using the Imbalance Principle we can thicken $N_{r+1}$ to have boundary slope $\infty'$.

It remains to show that we can achieve just two dividing curves for this $N_{r+1}$.  Note that $N_{r+1}$ is contained in a thickened torus $R$ representing $\breve{K}_r$ with $\partial R = T_2-T_1$ and where the dividing curves on $T_i$ have slope $q_{r+1}/p_{r+1}$.  The key now is that since $q_{r+1}/p_{r+1} \in (-1/A_r,0)$, there is twisting on both sides of $R$.  We can thus reduce the number of dividing curves on $N_{r+1}$ by either finding bypasses in $R\backslash N_{r+1}$ or by finding bypasses along $T_1$ or $T_2$ that can be extended into $R$, as in the proofs of Claims 4.1 and 4.3 in \cite{[EH1]}.

\textbf{Step 3:}  We now show that the $L_{r+1}$ at $\overline{tb}$ are distinguished by their rotation numbers.  To do this, we first note that since $q_{r+1}/p_{r+1} > -1/A_r$, there exists an integer $n \geq A_r$ with $-1/A_r \leq -1/n < q_{r+1}/p_{r+1} < -1/(n+1)$.  Changing to the standard $\mathcal{C}$ framing yields $-1/(n-A_r) < q_{r+1}/P_{r+1} < -1/((n+1)-A_r)$.  This thickened torus bounded by the tori with slopes $-1/(n-A_r)$ and $-1/((n+1)-A_r)$ is a universally tight {\em basic slice} in the sense of \cite{[H]}, and thus by an argument identical to that in Lemma 3.8 in \cite{[EH1]} we have that the set of rotation numbers achieved by $L_{r+1}$ at $\overline{tb}$ is:

\begin{equation}
\label{rlr+1}
r(L_{r+1}) \in \left\{\pm (P_{r+1} + q_{r+1}(n-A_r+r(L_r))) | tb(L_r)=A_r-n\right\}
\end{equation}

Changing to $p_j$'s and $q_j$'s yields:

\begin{equation}
r(L_{r+1}) \in \left\{\pm (p_{r+1} + nq_{r+1}+q_{r+1}r(L_r)) | tb(L_r)=A_r-n\right\}
\end{equation}

Now we know from the Legendrian classification of $\breve{K}_r$ that if $tb(L_r)=A_r-n$, then 

\begin{equation}
r(L_r) \in \left\{-(n-B_r), -(n-B_r)+2, \cdots, (n-B_r)-2,(n-B_r)\right\}
\end{equation}

Plugging these values of $r(L_r)$ just into the values $r(L_{r+1})=p_{r+1} + nq_{r+1}+q_{r+1}r(L_r)$ yields $r(L_{r+1})$ that begin on the left at $B_{r+1} < 0$, and then increase by $2q_{r+1}$, ending at $p_{r+1} + nq_{r+1}+q_{r+1}(n-B_r)$.  Reflecting these values across the $r=0$ axis yields the $r(L_{r+1})=-(p_{r+1} + nq_{r+1}+q_{r+1}r(L_r))$; these two distributions interleave to form one total distribution of $r$-values.  Thus, if we define $s=-p_{r+1}-nq_{r+1}$ we have that the distribution of $r(L_{r+1})$ when $\overline{tb}(L_{r+1})=A_{r+1}$ is as follows:

%\begin{equation}
%r(L_{r+1}) \in \left\{\pm(B_{r+1}+2\left\lfloor \frac{i}{2}\right\rfloor q_{r+1} +  [(-1)^{i+1}+(-1)^{2i}]s) | 0 \geq i \leq 
%\begin{equation}
%r(L_{r+1}) \in \left\{B_{r+1}, B_{r+1}+2s, B_{r+1}+2q_{r+1}, B_{r+1}+2q_{r+1}+2s, B_{r+1}+4q_{r+1}, B_{r+1}+4q_{r+1}+2s,\\

%\cdots,-(B_{r+1}+4q_{r+1}+2s),-(B_{r+1}+4q_{r+1}),-(B_{r+1}+2q_{r+1}+2s),-(B_{r+1}+2q_{r+1}),-(B_{r+1}+2s),-B_{r+1}\right\}
%\end{equation}
\begin{equation}
B_{r+1} < B_{r+1}+2s < B_{r+1}+2q_{r+1} <  \cdots < -(B_{r+1}+2q_{r+1}) < -(B_{r+1}+2s) < -B_{r+1}\nonumber
\end{equation}

Note that $q_{r+1} > s > 0$.  Algorithmically, the distribution of values for $r(L_{r+1})$ is achieved as follows:  begin on the left at $B_{r+1}$, and then move right to the next $r$-value by alternating lengths of $2s$ and $2(q_{r+1}-s)$, until one reaches $-B_{r+1}$.  As mentioned in \cite{[EH1]}, a way to see where these rotation numbers come from is noting that to each $L_r$ with $tb(L_r)=A_r-n$, there corresponds two $L_{r+1}^{\pm}$ at $\overline{tb}$, where $r(L_{r+1}^{\pm})=q_{r+1}r(L_r)\pm s$.  $L_{r+1}^{\pm}$ is obtained by removing a standard neighborhood of $N(S_{\pm}(L_r))$ from $N(L_r)$ and taking a Legendrian divide on a torus with slope $q_{r+1}/p_{r+1}$ inside $N(L_r) \backslash N(S_{\pm}(L_r))$.  Here $S_+$ indicates positive stabilization and $S_-$ means negative stabilization.  As a consequence, if $L_{r+1}$ and $L'_{r+1}$ are both at $\overline{tb}$ and have the same rotation number, then they must exist in basic slices that are associated to $L_r$ and $L'_r$ at $tb=A_r-n$ and having the same rotation number, as well as the same parity of stabilization for $L_r$ and $L'_r$.  These basic slices are thus contact isotopic since $\breve{K}_r$ is Legendrian simple, yielding a Legendrian isotopy from $L_{r+1}$ to $L'_{r+1}$ using a linearly foliated torus -- see Lemma 3.17 in \cite{[EH2]}.

\textbf{Step 4:}  We now show that if $tb(L_{r+1}) < \overline{tb}$, then $L_{r+1}$ destabilizes.  To see this, note that since $\overline{t}(K_{r+1})=0$, if $L_{r+1}$ has $tb(L_{r+1}) < \overline{tb}$, we know that $L_{r+1}$ is a Legendrian ruling on the boundary of a solid torus $N_r$ and that $N_r$ either contains a solid torus with $\textrm{slope}(\Gamma)=q_{r+1}/p_{r+1}$ or can be thickened to a solid torus with such a boundary slope, since $q_{r+1}/p_{r+1} > -1/A_r$.  Thus $L_{r+1}$ will destabilize by the Imbalance Principle.

\textbf{Step 5:}  We now show that if $L^v_{r+1}$ is in a valley of the Legendrian mountain range, that is $(r(L^v_{r+1})\pm 1,tb(L^v_{r+1}) + 1)$ have images in the mountain range, but $(r(L^v_{r+1}),tb(L^v_{r+1})+2)$ does not, then there are two Legendrian representatives of $K_{r+1}$ at $\overline{tb}$, namely the two closest peaks $L^+_{r+1}$ and $L^-_{r+1}$, such that $L^v_{r+1}=S^m_+(L^-_{r+1})=S^m_-(L^+_{r+1})$ for some $m>0$.  

To see this, first note that from the distribution of rotation numbers at $\overline{tb}$, there are two types of valleys, those with depth $s$, and those with depth $q_{r+1}-s$.  We first consider valleys of depth $s$.  Such a valley falls between two peaks represented by Legendrian knots at $\overline{tb}$, where $r(L^+_{r+1})=q_{r+1}r(L_r)+s$ and $r(L^-_{r+1})=q_{r+1}r(L_r)-s$.  So $r(L^v_{r+1})=q_{r+1}r(L_r)$ and $t(L^v_{r+1})=p_{r+1}+nq_{r+1}$; hence $L^v_{r+1}$ is a $(p_{r+1},q_{r+1})$ ruling on a standard neighborhood of $L_r$ where $t(L_r)=-n$.  Then we can stabilize $L_r$ both positively and negatively to obtain two different basic slices having boundary slopes $-1/n$ and $-1/(n+1)$.  In the one, there will be a boundary parallel torus with $t(L_{r+1})=0$ and a convex annulus that results in $s$ positive destabilizations of $L^v_{r+1}$; in the other there will be a convex annulus to a similar torus that results in $s$ negative destabilizations of $L^v_{r+1}$.

Now consider a valley of depth $q_{r+1}-s$.  Then such a valley falls between two peaks represented by $r(L^+_{r+1})$ and $r(L^-_{r+1})$ where $r(L^+_{r+1})=q_{r+1}r(L_r) - s$.  Thus $r(L^v_{r+1})=q_{r+1}(r(L_r)-1)$ and $t(L^v_{r+1})=-p_{r+1}-(n+1)q_{r+1}$; hence $L^v_{r+1}$ is a $(p_{r+1},q_{r+1})$ ruling on a standard neighborhood of $S_-(L_r)$.  Now note that if $r(L_r)=-(n-B_r)$, that would imply that $r(L^+_{r+1})=B_{r+1}$, which is not true.  Thus a consideration of the Legendrian mountain range for $\breve{K}_r$ allows us to conclude that $S_-(L_r)$ destabilizes both positively and negatively to obtain two different basic slices having boundary slopes $-1/n$ and $-1/(n+1)$.  In the one, there will be a boundary parallel torus with $t(L_{r+1})=0$ and a convex annulus that results in $q_{r+1} - s$ positive destabilizations of $L^v_{r+1}$; in the other there will be a convex annulus to a similar torus that results in $q_{r+1} - s$ negative destabilizations of $L^v_{r+1}$. 
\end{proof}

This proves Theorem \ref{main theorem 3}; a change of coordinates from $\mathcal{C}'$ to $\mathcal{C}$ then yields Corollary \ref{negativecablings}.

\bigskip\bigskip
\small{\scshape{University at Buffalo, Buffalo, NY}\\
\indent{\em E-mail address}:  \ttfamily{djl2@buffalo.edu}}\\\\\\

\end{document}